\documentclass[10pt,reqno]{amsart}
\usepackage{hyperref,amssymb,mathrsfs,graphicx,subfigure}
\usepackage{amsfonts,amssymb,amsmath,enumerate} 
\usepackage{amsthm}
\usepackage{graphicx,colortbl,xcolor}
\usepackage{floatrow}
\usepackage{soul}
\usepackage{comment}

\newtheorem{theorem}{Theorem}[section]
\newtheorem{lemma}{Lemma}[section]

\newtheorem{proposition}{Proposition}[section]
\newtheorem{remark}{Remark}[section]
\newtheorem{definition}{Definition}[section]

\DeclareMathOperator*{\essinf}{ess\,inf}

\newcommand{\bbt} {\mathbb T}

\def\BE{\begin{equation}}
\def\EE{\end{equation}}

\newcommand{\R}{\mathbb{R}}

\newcommand{\Z}{\mathbb{Z}}
\newcommand{\eps}{\varepsilon}

\newcommand{\vv}{\mathtt v}

\newcommand{\caG}{\mathcal{G}}

\makeatletter
\newcommand{\doublewidetilde}[1]{{%
  \mathpalette\double@widetilde{#1}%
}}
\newcommand{\double@widetilde}[2]{%
  \sbox\z@{$\m@th#1\widetilde{#2}$}%
  \ht\z@=.9\ht\z@
  \widetilde{\box\z@}%
}
\makeatother

\def\charf {\mbox{{\text 1}\kern-.30em {\text l}}}

\begin{document}

\title[Decay of entropy solutions to Euler-alignment systems]{%
Decay of periodic entropy solutions to Euler-alignment systems with non-constant kernel
}

\author
{Debora Amadori}
\address{\newline 
Dipartimento di Ingegneria e Scienze dell'Informazione e Matematica (DISIM), University of L'Aquila -- L'Aquila, Italy}
\email{debora.amadori@univaq.it}

\author
{Cleopatra Christoforou}
\address{\newline Department of Mathematics and Statistics, University of Cyprus -- Nicosia, Cyprus}
\email{christoforou.cleopatra@ucy.ac.cy}

\author
{Gianmarco Cipollone}
\address{\newline 
Dipartimento di Ingegneria e Scienze dell'Informazione e Matematica (DISIM), University of L'Aquila -- L'Aquila, Italy}
\email{gianmarco.cipollone@graduate.univaq.it}

\date{\today}



\thanks{\textbf{Acknowledgments.} The first and the third author acknowledge the support by the INdAM - GNAMPA Project with title “Analisi e controllo per alcuni problemi di evoluzione”, codice CUP E53C25002010001. They also acknowledge the support and the kind hospitality of the University of Cyprus, where part of this research was done. The second author also acknowledges the support by the Internal Grant with title ``Hyperbolic Conservation Laws: Theory and Applications" (HypConLaes) from University of Cyprus.}

\begin{abstract}
We consider a hydrodynamic model of flocking-type with pressure on the torus, with integrable interaction kernel 
and density bounded away from zero. We prove that, if an entropy weak solution exists, 
then its $L^2$ norm decays exponentially fast in time towards the mean values on the period. 
The proof relies on the study of a suitable energy functional that combines a strictly convex entropy for the system 
and a potential term, and this allows us to treat the nonlocal source term for a class of strictly positive convolution kernels in $L^1$.
\end{abstract}

\maketitle

\tableofcontents

\section{Introduction}
The mathematical modeling of self-organized systems is a topic that is receiving a lot of attention in the recent years, 
and it has brought new challenges in the area of partial differential equations that yield interesting questions in the mathematical community. 
Typical examples of self-organized systems can be met in biology: flocks of birds, swarms of bacteria or schools of fishes, but in general, 
self-organized system can be found in other sciences as long as emergent behavior is present in the system. Many of these mathematical models 
have been derived by the pioneering work of Cucker and Smale \cite{CuS1}. We refer to \cite{CFTV-2010, Shv2021} for general references on the subject.

In this paper we are interested in the Euler-type flocking system that takes the form
\begin{equation}\label{eq:system_Eulerian-rho-v}
    \begin{cases}
        \partial_t\rho +  \partial_x (\rho \vv) 
= 0, &\\
\partial_t (\rho\vv)
+  \partial_x \left(\rho\vv^2 + p(\rho) \right) =  \int_\bbt K(x-x')\rho(x)\rho(x')\left(\vv(x')  - \vv(x) \right)
\,dx',
    \end{cases}
\end{equation}
on the periodic domain $\bbt =\R/2\Z$, represented by the interval $[-1,1)$. 
Here, $(x,t)\in\bbt\times[0,+\infty)$, $\rho>0$ stands for the density, $\vv$ for the velocity 
and $p$ for the pressure. Our aim is to study the long term behavior of entropy weak solutions to \eqref{eq:system_Eulerian-rho-v} 
in the non-vacuum regime. We assume that the pressure term satisfies
\begin{equation}\label{hyp-on-p}
  p\in C^1([0,+\infty);[0,+\infty))\,,\qquad  p'>0
\end{equation}
and the convolution kernel $K$ satisfies
\begin{equation}\label{hyp-on-K}
K: \bbt\to \R\,,\qquad K \in L^1(\bbt)\,,\qquad K \mbox{ even}\,,\qquad \bar K:=\essinf_{x\in[-1,1)}K(x)>0\,.
\end{equation}
A class of interaction kernels $K$ that has been studied in \cite{ChHw2024} in the context of Euler-alignment system is 
\begin{equation*}
K(x) = |x|^{-\beta}\quad a.e.\ \mbox{on }\bbt,\quad \beta>0
\end{equation*}
and it is easy to check that for $0<\beta<1$ the kernels in this class satisfy conditions~\eqref{hyp-on-K}.
The constant case $K=\bar  K>0$ for all-to-all interaction models is also included in assumption~\eqref{hyp-on-K} and the existence of $BV$ entropy weak solutions is treated in~\cite{ACC2024}. However, for the non-constant case and under the assumption $p(\rho)= \kappa \rho$ with $\kappa>0$, the existence is under investigation. 

Concerning rigorous limit from kinetic formulation to hydrodynamic, we refer to \cite{ChHw2024,KMT15}
where the authors consider weak solutions to kinetic equations and their limit to suitably regular solutions to the hydrodynamic formulation. In \cite{ChHw2024} the authors rigorously derive the isentropic Euler-alignment system with singular communication weights from a BGK-alignment model in a multidimensional setting using a relative entropy approach. In particular the result on the one dimensional torus holds with $p(\rho)=\rho^3$ 
and with a singular kernel of the form $K(x)=|x|^{-\beta}$ with $\beta\in(0,\tfrac{3}{2})$, while in \cite{KMT15} the authors rigorously derive the Euler-flocking system with isothermal pressure, and smooth and symmetric kernel, again in a multidimensional setting and by means of relative entropy techniques. For general references on hydrodynamic models, see \cite{HT08,Tadmor2023}. 

Having set $m:=\rho\vv$ as the momentum, we rewrite the system for the conserved quantities as
\begin{equation}\label{eq:system_Eulerian-p}
    \begin{cases}
        \partial_t\rho + \partial_x m =0,\\
        \partial_t m + \partial_x \left( \frac{m^2}{\rho} + p(\rho) \right) = \int_\bbt K(x-x')(\rho(x)m(x')-\rho(x')m(x))\,dx'
    \end{cases}
\end{equation}
and consider the initial condition
\begin{equation}\label{eq:init-data}
    (\rho,m)(x,0) = (\rho_0(x),m_0(x))\,,\qquad x\in\bbt\,.
\end{equation}

In view of studying the long term behavior of entropy weak solutions to \eqref{eq:system_Eulerian-p}-\eqref{eq:init-data}, 
we assume that an entropy weak solution $(\rho, m)$ to \eqref{eq:system_Eulerian-p}-\eqref{eq:init-data} exists up to some time $T^*$, with $T^* \leq +\infty$ and that the density $\rho$ is uniformly bounded and away from zero, i.e. there are positive constants $0<\bar\rho_{inf}\le\bar\rho_{sup}$ such that
\begin{equation}\label{eq:bounds-rho}
    \bar\rho_{inf} \leq \rho(x,t) \leq \bar\rho_{sup}\qquad (x,t)\in \Omega:= \bbt\times [0,T^*)\,.
\end{equation}
We remark that, if the convolution kernel is constant, the result in \cite{ACC2024} captures 
the global existence of periodic entropy weak solutions, for periodic initial data having finite total variation and initial density bounded away from zero, in the case of linear pressure term 
while the long-time behavior is not treated therein. 
The 
existence result to the case of an integrable, positive convolution kernel is treated in a forthcoming publication by the authors. We also refer to \cite{AC_2022,AC2024} for other results on the Euler-alignment system with linear pressure. In \cite{AC_2022} the authors study the existence of solutions in the case of a constant convolution kernel, for compactly supported initial data with density uniformly positive away from zero. Moreover, they analyze the large time-behavior of solutions under a suitable condition on the initial data. However, in \cite{AC2024}, this condition is removed and the unconditional time asymptotic flocking is proved. 

In this paper, our approach to study the long-time behavior of weak solutions is inspired by the work of Dafermos to hyperbolic systems of balance laws with dissipative source. In particular we adopt the strategy in \cite[Section 3]{Dafermos_periodic2015} to deal with our system, where the presence of the nonlocal source term requires a careful treatment. 

Here is the main result of this paper.

\begin{theorem}\label{thm 1:L2-decay}
Assume $(\rho_0,m_0)\in L^\infty(\bbt)$ 
and that $\rho_0(x)\in [\bar\rho_{inf},\bar\rho_{sup}]$ for a.e. $x\in\bbt$, for some $0< \bar\rho_{inf}\le \bar\rho_{sup}$ and
suppose that ~\eqref{hyp-on-p}, \eqref{hyp-on-K} hold. Let $T^*>0$ or $T^*=+\infty$.

If $(\rho, m)$ is an entropy weak solution of the problem \eqref{eq:system_Eulerian-p}-\eqref{eq:init-data} on $\bbt\times[0,T^*)$ that satisfies~\eqref{eq:bounds-rho}, 
then there exist positive constants $\mu$ and $c$, independent of $T^*$, such that 
    \begin{equation}\label{S1:decay-bis} \| \rho(\cdot,t) -\bar\rho\|_{L^2(\bbt)}+ \| m(\cdot,t)-\bar m \|_{L^2(\bbt)}
    \leq \mu\left[\| \rho_0-\bar\rho\|_{L^2(\bbt)}+\| m_0-\bar m\|_{L^2(\bbt)}\right]e^{-ct}
    \end{equation}
for all $t\in [0,T^*)$, where $\bar\rho$ and $\bar m$ are the average values on $\bbt$ of $\rho_0$ and $m_0$, respectively.
\end{theorem}
We remark that, if $T^*=+\infty$ in Theorem \ref{thm 1:L2-decay}, i.e. the solution exists globally in time, 
then \eqref{S1:decay-bis} implies that the model exhibits a velocity alignment property. Indeed, we immediately deduce that
    \begin{equation}\label{S1:alignment}
    \| \vv(\cdot,t)-\bar \vv\|_{L^2(\bbt)}\rightarrow 0,\qquad\text{as } t\to+\infty,     \qquad 
    \bar \vv = \frac{\bar m}{\bar\rho}\,.
    \end{equation} 
   
Several contributions on the subject are devoted to suitably regular solutions; for instance in \cite{Choi2019} the author studies the well-posedness and the long-time behavior of strong solutions for the compressible isothermal Euler system with nonlocal dissipation on a periodic domain in a multidimensional setting. This result is obtained under suitable smallness and regularity assumptions on the initial data.

Concerning periodic solutions of systems of balance laws, in \cite[Section 12.10]{Dafermosbook} the large-time behavior of solutions to genuinely nonlinear, $2\times2$ systems of conservation laws is studied, and the total variation is shown to decay as $1/t$. In \cite{Dafermos_periodic2015,{D-BAcSin-2016}} strictly hyperbolic systems — including the compressible Euler system with partially dissipative source terms — are analyzed under general pressure laws and small initial data, and exponential decay of the total variation is established. 

In addition, systems of conservation laws with a non-local source term have been studied in \cite{{CG2007},{CG2008}}, 
where the well-posedness of the Cauchy problem for small BV data  has been addressed by means of a semigroup approach.

Furthermore, we observe that most of the research on hydrodynamic models for flocking is devoted to the pressureless Euler-alignment system, see \cite{Shv2021} and references therein. In particular, in \cite{DS-2021} the pressureless analogue of \eqref{eq:system_Eulerian-p} is considered and analysis is conducted by means of an alignment functional which is similar, in structure, to the functional ${\mathcal H}(t)$ defined at \eqref{def:H-function}.

The strategy of the paper is as follows. We first rewrite system \eqref{eq:system_Eulerian-rho-v} in terms of the perturbation variables $(\tilde\rho,\tilde m)$ 
that represent the deviation from the average values. We then construct a strictly convex entropy $\tilde\eta(\tilde\rho,\tilde m)$ inspired by the general structure 
of relative entropy, and we introduce the potential functional
\begin{align*}
    \Phi(x,t)  & := \int_0^x \tilde\rho(y,t)\,dy - \frac{1}{2}\int_{-1}^1\int_0^z \tilde\rho(y,t)\,dy\,dz\;,
    \end{align*}
which is 2-periodic and with zero average on the period and it can be controlled by the $L^2$ norm of $\tilde\rho$. Using the entropy $\tilde \eta$ and the potential functional $\Phi$, we define the energy functional $\mathcal{E}(t)$:
\begin{align*}
    \mathcal{E}(t) & := \int_\bbt \left[ \sigma\tilde\eta + \Phi^2 - \frac{1}{\bar K\bar\rho}\Phi \tilde m \right]\,dx 
    \end{align*}
with $\sigma>0$ sufficiently large. 

The strict convexity of $\tilde\eta$, together with the bound on $\Phi$, allows us to show that $\mathcal{E}(t)$ is equivalent to the squared $L^2$ norm of the perturbations $(\tilde\rho,\tilde m)$. We further prove that $\mathcal{E}(t)$ decays exponentially in time, from which we deduce the convergence in $L^2$ of the solution $(\rho,m)$ toward the equilibrium state determined by the spatial averages of the initial data $(\rho_0,m_0)$.

\smallskip
The structure of the paper is the following.

- In Section~\ref{S2}, we change the variables ($\rho,m)$ to  $(\tilde \rho,\tilde m)$ that are the perturbations of the density and the momentum with respect to their averages $\bar\rho$ and $\bar m$, respectively and rewrite the system for new variables.
   After having given the definition of entropy weak solution, we introduce the relative entropy $\tilde\eta$ in terms of the new variables, the energy functional $\mathcal{E}(t)$ and other functionals that are later employed in the proofs. Last, we restate Theorem~\ref{thm 1:L2-decay} in terms of the new variables, see Theorem~\ref{th:L2-decay}, and state the decay of the $L^2$ norm of the solutions for the shifted system.
    
- In Section \ref{S3}, we analyze the energy functional $\mathcal{E}(t)$ in more detail.
    Thanks to the analysis in Section~\ref{S2} and more specifically, the estimates on the functionals therein and the relative entropy, we prove the decay in time of the energy functional 
    $\mathcal{E}(t)$ and then complete the proof of Theorem~\ref{th:L2-decay}
    that implies immediately Theorem \ref{thm 1:L2-decay}.

\section{Set Up of the problem}\label{S2} 
\setcounter{equation}{0}

We consider 2-periodic initial data $U_0(x)=(\rho_0(x),\,m_0(x))$ and define 
\begin{equation}\label{S2: initial average}
\bar \rho :=\frac 12 \int_{\bbt} \rho_0(x)\,dx\,,\qquad {\bar m}:= \frac{1}{2}\int_{\bbt} m_0(x)\,dx
\end{equation}
where $\bar \rho\in[\bar\rho_{inf},\bar\rho_{sup}]$. Without loss of generality $\bar m$ can be assumed to be
zero, since with a change of coordinates we can reduce the problem to zero average momentum. 
Indeed one can apply the change of the variables $(x,t)$ and $\vv$ given by 
$x\mapsto x-\bar\vv t$, $\vv\mapsto\vv-\bar\vv$
with $\bar\vv$ the average velocity as given in \eqref{S1:alignment}.

For convenience, we use the following notation for the source term of system~\eqref{eq:system_Eulerian-p}: 
\begin{equation}\label{def:G-periodic}
    \mathcal{G}[\rho,m](x) := \int_{\bbt} K(x-x')\left(\rho(x)m(x')  - \rho(x')m(x) \right)\,dx',
\end{equation}
from here and on to simplify the forthcoming analysis. 

We also note that an entropy-entropy flux pair $(\eta,q)$ for system~\eqref{eq:system_Eulerian-p} 
is a smooth pair of functions $\eta=\eta(\rho,m)$ and $q=q(\rho,m)$, with $\eta$ convex, that satisfies the equations
$$ \left(-\frac{m^2}{\rho^2}+p'\right) \eta_m=q_\rho,\qquad \eta_\rho+\frac{2m}{\rho} \eta_m =q_m\,.
$$
Now we provide the definition of an \emph{entropy weak solution} $(\rho, m)(x,t)$ for $x\in\bbt$ and $t\in[0,T^*)$. 

\begin{definition}
Assume that $(\rho_0,m_0)\in L^\infty(\bbt)$ with $\rho_0(x)>0 $ for a.e. $x\in\bbt$.
Assume that $K\in L^1(\bbt)$. Let $T^*>0$ or $T^*=+\infty$.  
      
        We say that $(\rho,\,m)$ is an entropy weak solution of \eqref{eq:system_Eulerian-p}-
        \eqref{eq:init-data} on $\Omega = \bbt\times [0,T^*)$ if:
    \begin{itemize}
\item $(\rho,m)(\cdot,0)=(\rho_0,m_0)$ in $L^1(\bbt)$;
\item $L^1$-Lipschitz continuous dependence in time holds true, i.e. for $ U = (\rho,\, m)$,
        $$
            \|  U(t_2)- U(t_1)\|_{L^1(\bbt)} \leq L|t_2-t_1|,\qquad t_1,t_2\in [0,T^*)
            \,,
      $$ for some positive constant $L$;
\item $(\rho,m)$ is a solution to~\eqref{eq:system_Eulerian-p} in the sense of distributions on $\bbt\times (0,T^*)$;
\item for every pair of smooth entropy-entropy flux $(\eta,q)$, it holds
  \begin{equation*}
        \partial_t\eta(\rho,m)+\partial_x q(\rho,m) \leq \eta_m \, \mathcal{G}[\rho,m],
    \end{equation*}
    in the sense of distributions on $\bbt\times (0,T^*)$.
        \end{itemize}
\end{definition}

Using the symmetry of the kernel $K$, it is easy to verify that entropy weak solutions $(\rho,m)$ to~\eqref{eq:system_Eulerian-p} 
on $\bbt\times[0,T^*)$ conserve mass and momentum, that is,
\begin{equation*}
    \frac{1}{2}\int_\bbt \rho(x,t)\,dx=\bar\rho, \qquad \frac{1}{2}\int_\bbt m(x,t)\,dx=0\qquad \forall\, t\in [0,T^*).
\end{equation*}
It should be pointed out that the sense of distributions in the definition above requires that the corresponding relations hold for every test function $\phi\in C^1(\bbt\times (0,T^*))$ with support contained in $\bbt\times [T_1,T_2]$ for some $0<T_1<T_2<T^*$. In the case of the entropy inequality, one, of course, also needs $\phi\ge 0$.
    
The analysis lies on studying the evolution equations of the shifted variables $(\tilde\rho, \tilde m)$ defined by
\begin{equation}\label{def:tilde-var}
    \tilde\rho(x,t):=\rho(x,t)-\frac{1}{2}\int_{\bbt}\rho_0(x)\,dx\,,\qquad \tilde m(x,t):=m(x,t)-\frac{1}{2}\int_{\bbt}m_0(x)\,dx\,.
\end{equation}
The second identity in \eqref{S2: initial average} yields that $\tilde m(x,t) =m(x,t)$.

The new variables $\tilde \rho$, $\tilde m$ have both zero average over $\bbt$, i.e.
\begin{equation}\label{eq:integrals=0}
        \int_\bbt \tilde\rho(x,t)\,dx = 0,\qquad 
        \int_\bbt \tilde m(x,t)\,dx = 0,\qquad t\in[0, T^*)\,.
    \end{equation}
Then direct computations yield the system for the new variables $(\tilde\rho,\tilde m)$, that is
\begin{equation}\label{eq:intermedia system}
    \begin{cases}
\partial_t\tilde\rho+\partial_x\tilde m = 0,\\
        \partial_t \tilde m + \partial_x \left(\frac{\tilde m^2}{\tilde\rho+\bar\rho}+p(\tilde\rho+\bar\rho)\right) = \int_{\bbt}K(x-x')\Big((\tilde\rho(x)+\bar\rho)\tilde m(x')-\tilde m(x)(\tilde\rho(x')+\bar\rho)\Big)\,dx'.
    \end{cases}
\end{equation}
In a similar fashion, we define the shifted kernel $\tilde K(x) := K(x) - \bar K$ and from assumptions~\eqref{hyp-on-K}, we know that 
\begin{equation}\label{hyp-on-K tilde}
\tilde K: \bbt\to \R\,,\qquad \tilde K \in L^1(\bbt)\,,\qquad \tilde K \mbox{ even}\,,\qquad \tilde K(x)\ge 0\,.
\end{equation}
Hence, the source function from~\eqref{def:G-periodic} takes the form
\begin{equation*}
    \caG[\rho,m](x) = -2\bar K\bar\rho \, \tilde m(x) + \tilde \caG[\tilde\rho+\bar\rho, \tilde m](x)\,,
\end{equation*}
with
\begin{equation}\label{def:G-tilde}
    \tilde \caG[\rho,m](x):=\int_{\bbt} \tilde K (x-x')(\rho(x)m(x')-\rho(x')m(x))\,dx'.
\end{equation}
Indeed,
\begin{align*}
   \caG[\rho,m](x) = & \int_{\bbt}K(x-x')((\tilde\rho(x)+\bar\rho)\tilde m(x')-\tilde m(x)(\tilde\rho(x')+\bar\rho))\,dx'\\
    = & \int_{\bbt}\bar K \left[(\tilde\rho(x)+\bar\rho)\tilde m(x')-\tilde m(x)(\tilde\rho(x')+\bar\rho)\right]\,dx' \\
     &+\int_{\bbt}\tilde K(x-x')\left[(\tilde\rho(x)+\bar\rho)\tilde m(x')-\tilde m(x)(\tilde\rho(x')+\bar\rho)\right]\,dx'\\
     = & \int_\bbt \bar K (\tilde\rho(x)+\bar\rho)\tilde m(x')\,dx' - \tilde m(x)\left[\int_\bbt \bar K \tilde\rho(x')\,dx'+2\bar K \bar\rho\right] + \tilde \caG[\tilde\rho+\bar\rho,\tilde m](x)\\
     = & -2\bar K\bar\rho\tilde m(x) + \tilde \caG[\tilde\rho+\bar\rho,\tilde m](x)\;,
\end{align*}
since~\eqref{eq:integrals=0} holds true. Using these calculations, system~\eqref{eq:intermedia system} reads as
\begin{equation}\label{eq:final-system}
    \begin{cases}
        \partial_t\tilde\rho+\partial_x \tilde m=0,\\
        \partial_t \tilde m+ \partial_x \left(\frac{\tilde m^2}{\tilde\rho+\bar\rho}+p(\tilde\rho+\bar\rho)\right) = -2\bar K\bar\rho \, \tilde m(x) + \tilde \caG[\tilde \rho+\bar\rho, \tilde m]
    \end{cases}
\end{equation}
 and we assign the corresponding initial data
\begin{equation}\label{eq:init-data tilde}
(\tilde\rho,\tilde m)(x,0)=\left(\tilde\rho_0(x), \tilde m_0(x)\right)\,\qquad x\in\bbt\;,
\end{equation}
 with $\tilde\rho_0(x)=\rho_0(x)-\bar\rho$ and $\tilde m_0(x)=m_0(x)$, for $x\in\bbt$. From~\eqref{eq:bounds-rho}, the shifted density $\tilde \rho$ satisfies 
\begin{equation}\label{eq:bounds-rho tilde}
  - \bar\rho<\bar\rho_{inf}-\bar\rho \le \tilde\rho(x,t)\le \bar\rho_{sup} -\bar\rho
\end{equation}
 for $x\in\bbt$ and $t\in[0,T^*)$, and $\bar\rho\in [\bar\rho_{\inf},\bar\rho_{sup}]$ is given at~\eqref{def:tilde-var}.

The above change of variables implies that $(\rho,m)$ is an entropy weak solution to \eqref{eq:system_Eulerian-p},~\eqref{eq:init-data} on $\Omega$ 
if and only if $(\tilde\rho, \tilde m)$ is an entropy entropy weak solution to~\eqref{eq:final-system},~\eqref{eq:init-data tilde} on $\Omega$. 

\begin{definition}\label{def:2.2}
  Let $K\in L^1(\bbt)$. Let $\bar\rho>0$, $T^*>0$  and $\Omega$ as in \eqref{eq:bounds-rho}.
  Assume $(\tilde\rho_0,\tilde m_0)\in L^\infty(\bbt)$, with 
  $\tilde\rho_0(x)+\bar\rho>0$ for a.e. $x\in\bbt$, and 
  \begin{equation*}\label{eq:integrals=0 init}
        \int_\bbt \tilde\rho_0(x)\,dx =\int_\bbt \tilde m_0(x)\,dx = 0\,.
    \end{equation*}
   We say that $(\tilde\rho,\, \tilde m)$ is an entropy weak solution to the problem \eqref{eq:final-system}--\eqref{eq:init-data tilde} 
   on $\Omega$ if:
    \begin{itemize}
        \item $(\tilde\rho,\tilde m)(\cdot,0)=(\tilde\rho_0,\tilde m_0)$ in $L^1(\bbt)$;
        \item  $L^1$-Lipschitz continuous dependence in time holds true, i.e. for $\tilde U = (\tilde\rho,\tilde m)$,
        $$
            \| \tilde U(t_2)-\tilde U(t_1)\|_{L^1(\bbt)} \leq L|t_2-t_1|,\qquad t_1,t_2 \in [0,T^*)
      $$  for some positive constant $L$\,;
        \item $(\tilde\rho,\tilde m)$ is a solution to \eqref{eq:final-system} in the sense of distributions;
        \item for every pair of smooth entropy-entropy flux $(\tilde\eta,\tilde q)$ 
        with $\tilde\eta$ convex, it holds
    \end{itemize}
    \begin{equation*}
        \partial_t\tilde\eta(\tilde\rho, \tilde m)+\partial_x\tilde q(\tilde\rho,\tilde m) \leq \frac{\tilde m}{\tilde \rho+\bar\rho}\left( -2\bar K\bar\rho\, \tilde m + \tilde \caG[\tilde\rho+\bar\rho,\tilde m] \right),
    \end{equation*}
    in the sense of distributions on $\bbt\times(0,T^*)$.
\end{definition}
We state the main result for the problem \eqref{eq:final-system}--\eqref{eq:init-data tilde}.
\begin{theorem}\label{th:L2-decay}  
Assume \eqref{hyp-on-p}, \eqref{hyp-on-K tilde} and that $\bar K>0$, $\bar\rho>0$. Furthermore, assume there exists an entropy weak solution $(\tilde\rho,\tilde m)(\cdot,t)$ to problem \eqref{eq:final-system}--\eqref{eq:init-data tilde} for $t\in [0,T^*)$ that satisfies~\eqref{eq:bounds-rho tilde} on $\Omega$ for some $\bar\rho_{inf}<\bar\rho<\bar\rho_{sup}$. Then there exist positive constants $\mu$ and $c$ such that
    \begin{equation}\label{prop:three}
         \| (\tilde \rho,\tilde m)(\cdot,t) \|_{L^2(\bbt)} \leq \mu\, e^{-ct}\, \| (\tilde \rho_0,\tilde m_0)\|_{ L^2(\bbt)},\qquad \forall \, t\in [0,T^*)
    \end{equation}
    with $\mu$ and $c$ independent of time. 
\end{theorem}

The analysis relies on the study of Lyapunov-type functionals that allow us to control the energy functional 
(see Subsec.~\ref{sec:2.1}) which is equivalent to the $L^2$ norm of the solution. 
The exponential decay in time of the energy functional yields the $L^2$ decay of the entropy weak solution on $\bbt$.

\subsection{The functionals}\label{sec:2.1}
This subsection introduces the functionals used to capture the decay behavior of an entropy weak solution $(\tilde\rho,\tilde m)$ 
to \eqref{eq:final-system}–\eqref{eq:init-data tilde} on $\Omega$ with $\bar\rho>0$.

\begin{enumerate}
\item The functional that measures the contribution of the convolution kernel $\tilde K$ is 
\begin{align}
    \mathcal{H}(t) & := \frac{1}{2}\int_\bbt\int_\bbt \tilde K(x-x')\left[\tilde m(x,t)\sqrt{\frac{\tilde\rho(x',t)+\bar\rho}{\tilde\rho(x,t)+\bar\rho}}-\tilde m(x',t)\sqrt{\frac{\tilde\rho(x,t)+\bar\rho}{\tilde\rho(x',t)+\bar\rho}} \,\right]^2dx\,dx'\,.\label{def:H-function}
    \end{align}
 \item The potential functional $\Phi$ is defined by
    \begin{align}
    \Phi(x,t)  & := \int_0^x \tilde\rho(y,t)\,dy - \frac{1}{2}\int_{-1}^1\int_0^z \tilde\rho(y,t)\,dy\,dz\;.\label{def:Phi-function}
    \end{align}
\item The convex entropy $\tilde\eta(\tilde\rho,\tilde m)$ of the system \eqref{eq:final-system} is given by
\begin{equation}\label{def:entropy}
    \tilde\eta(\tilde\rho,\tilde m)= (\tilde\rho+\bar\rho)(\eps(\tilde\rho+\bar\rho)-\eps(\bar\rho))-\tilde\rho\,\bar\rho\eps'(\bar\rho)+\frac{1}{2}\frac{\tilde m^2}{\tilde\rho+\bar\rho}\;,
\end{equation}
and it is normalized as follows:
\begin{equation*}
            \tilde\eta(0,0)=0,\quad \nabla\tilde\eta(0,0)=0\,.
        \end{equation*}
        We will see that it is strictly convex and there are $0<c_1<c_2$ constants such that
        \begin{equation*}
         c_1(\tilde\rho^2+\tilde m^2)\leq \tilde\eta(\tilde\rho,\tilde m)\leq c_2(\tilde \rho^2+\tilde m^2)\,,
        \end{equation*}
see the forthcoming \eqref{def:c1c2}.
\item The energy functional $\mathcal{E}(t)$ is introduced by means of the 
entropy $\tilde\eta$ and the potential functional $\Phi$,
    \begin{equation} \label{def:energy-function}
    \mathcal{E}(t) := \int_\bbt \left[ \sigma\tilde\eta + \Phi^2 - \frac{1}{\bar K\bar\rho}\Phi \tilde m \right]\,dx 
    \end{equation}
where $\sigma>0$ is a constant to be determined. 

\item Last, we need a modified kinetic energy functional $\mathcal{Z}(t)$ given by
    \begin{align}
        \mathcal{Z}(t) & := \left(\sigma\kappa-\frac{1}{\bar K}\right)\int_\bbt \frac{\tilde m^2}{\tilde\rho+ \bar\rho}\,dx + \beta \int_\bbt \tilde\rho^2\,dx,\label{def:Z-function}
\end{align}
 and the positive constants $\kappa$ and $\beta$ satisfy the inequalities
\begin{equation}\label{kappa-bound}
 0<   \kappa < 2\bar K\bar\rho\,,\qquad 0<\beta<  \frac{C}{\bar K\bar\rho}\,,
\end{equation}
with $C:= \inf_{r\in (\bar\rho_{inf},\bar\rho_{sup})} p'(r) >0$, which holds due to~\eqref{hyp-on-p} and~\eqref{eq:bounds-rho}.
\end{enumerate}
As it is shown in the forthcoming analysis, the constant $\mu$ of Theorem~\ref{th:L2-decay} satisfies 
$\mu^2=c_2\sigma + 2 + \frac{1}{\bar K\bar \rho}$, while $c_2$ and $c$ are given at~\eqref{def:c1c2} and \eqref{eq:constants-c} respectively 
and do not depend on the bounds on the momentum.
Hence, both $\mu$ and $c$ depend on the $L^\infty$ bound of the density, but not of the momentum.

In the following subsection, we study the potential functional $\Phi$ and the nonlocal functional $\mathcal{H}$ using relative entropy techniques and these results are used in Section~\ref{S3} to prove Theorem~\ref{th:L2-decay}.

\subsection{Estimates on the potential functional}\label{subsec:2.2}
To begin with, we prove an identity for the functional $\mathcal{H}(t)$ given at \eqref{def:H-function} 
in terms of the nonlocal source term $\tilde \caG$ given at~\eqref{def:G-tilde}.

\begin{proposition}\label{prop:property-G}
    Let $\bar\rho>0$ and assume $\tilde K$ satisfies properties~\eqref{hyp-on-K tilde}. 
    Then 
 for $\mathcal{H}(t)$ defined in \eqref{def:H-function},
 the following identity holds:
 \begin{equation}\label{eq:property-G}
        \mathcal{H}(t)=-\int_\bbt \frac{\tilde m(x)}{\tilde \rho(x)+\bar\rho}\tilde\caG[\tilde\rho+\bar\rho, \tilde m]\,dx \ge 0\;,
\end{equation}
for all $t\in[0,T^*)$ with $(\tilde\rho,\tilde m )\in L^1(\bbt)$.
\end{proposition}
\begin{proof} Thanks to the symmetry of the kernel $\tilde K$, we perform the following computations
    \begin{align*}
   - \mathcal{H}(t) =& \int_\bbt\frac{m(x)}{\rho(x)}  \tilde\caG[\rho,m]\, dx  = \int_\bbt\int_\bbt \tilde K(x-x')\frac{m(x)}{\rho(x)}[\rho(x)m(x')-\rho(x')m(x)]\,dx'dx\\
   = & \int_\bbt\int_\bbt \tilde K(x-x')\frac{m(x')}{\rho(x')}[\rho(x')m(x)-\rho(x)m(x')]\,dx\,dx'\\
   = & \frac{1}{2}\int_\bbt\int_\bbt  \tilde K(x-x')\Big[ \frac{m(x)}{\rho(x)}(\rho(x)m(x')-\rho(x')m(x))\\
    &\quad + \frac{m(x')}{\rho(x')}(\rho(x')m(x)-\rho(x)m(x'))\Big]\,dx\,dx'\\
   = & \frac{1}{2}\int_\bbt\int_\bbt \tilde K(x-x')\left[ 2m(x)m(x')-m(x)^2\frac{\rho(x')}{\rho(x)}-m(x')^2\frac{\rho(x)}{\rho(x')}\right]dx\,dx'\\
   = & -\frac{1}{2}\int_\bbt\int_\bbt \tilde K(x-x')\left[m(x)\sqrt{\frac{\rho(x')}{\rho(x)}}-m(x')\sqrt{\frac{\rho(x)}{\rho(x')}} \,\right]^2dx\,dx'
   \leq 0.
\end{align*}
The last inequality is true since $\tilde K$ is positive and employing definition~\eqref{def:H-function} for $\rho=\tilde\rho+\bar\rho$ and $m=\tilde m$. The proof is complete.
\end{proof}

We proceed with a preliminary analysis of the potential functional $ \Phi$ given at~\eqref{def:Phi-function}. For later convenience, we observe that $ \Phi$ can be rewritten as
\begin{equation}\label{def:potential_2}
    \Phi(x,t):= \psi(x,t)- <\psi(\cdot,t)>\,,      
\end{equation}
where 
\begin{equation} \psi(x,t) := \int_0^x\rho(y,t)\,dy\,,\quad x\in\bbt\,;\qquad
<\psi(\cdot,t)>:= \frac{1}{2}\int_{-1}^{1}\psi(z,t)\,dz    \,.
\end{equation}
Then we immediately deduce the equations:
\begin{align}
    \partial_x\Phi(x,t) & = \tilde\rho(x,t),\label{eq:partial-x-phi}\\
    \partial_t\Phi(x,t) & = - \tilde m(x,t).\label{eq:partial-t.phi}
\end{align}
Since $\int_{\bbt} \tilde m=0$ and $\rho$ is 2-periodic in space, then also $\Phi$ is 2-periodic in space.
Indeed, for any $x\in\bbt$,
\begin{align*}
    \Phi(x+1,t)  -\Phi(x-1,t)& =\int_0^{x+1}\tilde\rho(y,t)\,dy - \int_0^{x-1}\tilde\rho(y,t)\,dy\\
    &\qquad-\frac{1}{2}\int_\bbt\int_0^z\tilde\rho(y,t)\,dy\,dz + \frac{1}{2}\int_\bbt\int_0^z \tilde\rho(y,t)\,dy\,dz\\
    &   = \int_{x-1}^{x+1}\tilde\rho(y,t)\,dy = \int_\bbt\tilde \rho(y,t)\,dy = 0.
\end{align*}
It is also straightforward to verify that the average of $\Phi$
over a period is zero. Furthermore, we claim that
\begin{equation}\label{eq:bounds-phi}
    \int_{-1}^1 \Phi^2(x,t)\,dx \leq 2\int_{-1}^1\tilde \rho^2(x,t)\,dx\;,\qquad\forall t\in[0,T^*)\,.
\end{equation}
To prove the claim, we first rewrite $\Phi(x,t)$ as $$\Phi(x,t)=\int_{\bar x}^x \tilde\rho(y,t)\,dy$$ with $\bar x\in[-1,1)$ satisfying
$\int_0^{\bar x} \tilde\rho(y,t)\,dy = \frac{1}{2}\int_{-1}^1\int_0^z\tilde\rho(y,t)\,dydz$ and then estimate
\begin{align*}
    \int_{-1}^1 \Phi^2(x,t)\,dx & = \int_{-1}^1 \left( \int_{\bar x}^x\tilde\rho(y,t)\,dy \right)^2dx \leq \int_{-1}^1 \left[ \left( \int_{\bar x}^x1\,dy \right)^{\frac{1}{2}}\left( \int_{\bar x}^x\tilde\rho^2(y,t)\,dy \right)^{\frac{1}{2}} \right]^2dx \\
    &\quad = \int_{-1}^1 \left( \int_{\bar x}^x1\,dy \right)\left( \int_{\bar x}^x\tilde\rho^2(y,t)\,dy \right)dx \leq \int_{-1}^1 \left| \int_{\bar x}^x1\,dy \right|\left( \int_{-1}^1\tilde\rho^2(y,t)\,dy \right)dx \\
    &\quad = \left( \int_{-1}^1\tilde\rho^2(y,t)\,dy \right) \int_{-1}^1 \left| \int_{\bar x}^x1\,dy \right|dx = \left( \int_{-1}^1\tilde\rho^2(y,t)\,dy \right) \int_{-1}^1 \left| x-\bar x \right|\,dx\\
    &\quad = \left(\bar x^2+1\right)\left( \int_{-1}^1\tilde\rho^2(y,t)\,dy \right) 
    \,.
\end{align*}
Since $|\bar x|\le 1$, then claim~\eqref{eq:bounds-phi} follows.

\subsection{The entropy inequality} 
In this subsection we consider the entropy functional \eqref{def:entropy} and investigate its properties.

By~\eqref{hyp-on-p}, the system~\eqref{eq:final-system}  is endowed with a uniformly convex entropy function $\tilde\eta(\tilde\rho,\tilde m)$ normalized by $\tilde\eta(0,0)=0$ and $\nabla\tilde\eta(0,0)=0$, given by 
\begin{equation*}
    \tilde\eta(\tilde\rho,\tilde m)= (\tilde\rho+\bar\rho)(\eps(\tilde\rho+\bar\rho)-\eps(\bar\rho))-\tilde\rho\,\bar\rho\eps'(\bar\rho)+\frac{1}{2}\frac{\tilde m^2}{\tilde\rho+\bar\rho}\;,
\end{equation*}
where $\eps(\omega)$ is a function  that satisfies $\eps'(\omega)=\frac{p(\omega)}{\omega^2}$.

The gradient of the entropy is equal to
\begin{equation*}
    \nabla\tilde\eta(\tilde\rho,\tilde m)=\left(\eps(\tilde\rho+\bar\rho)-\eps(\bar\rho)+(\tilde\rho+\bar\rho)\eps'(\tilde\rho+\bar\rho)-\bar\rho\eps'(\bar\rho)-\frac{1}{2}\frac{\tilde m^2}{(\tilde\rho+\bar\rho)^2},\,\frac{\tilde m}{\tilde\rho+\bar\rho}\right),
\end{equation*}
while the Hessian matrix is equal to
\begin{equation*}
    \nabla^2{\tilde\eta}(\tilde\rho,\tilde m) = 
    \begin{bmatrix}
        (\tilde\rho+\bar\rho)p'(\tilde\rho+\bar\rho) + \dfrac{\tilde m^2}{(\tilde\rho+\bar\rho)^3} & -\dfrac{\tilde m}{(\tilde\rho+\bar\rho)^2}\\
        -\dfrac{\tilde m}{(\tilde\rho+\bar\rho)^2} & \dfrac{1}{\tilde\rho+\bar\rho}
    \end{bmatrix},
\end{equation*}
with determinant
$$\det\nabla^2{\tilde\eta}= p'+2\frac{\tilde m^2}{(\tilde\rho+\bar\rho)^4}\;.
$$ 
Using~\eqref{hyp-on-p} and~\eqref{eq:bounds-rho tilde}, we obtain that $\tilde\eta$ is strictly convex. Then, there exist two constant values $0<c_1<c_2$, depending on the range of values of $\tilde \rho$ and $\tilde m$, such that
\begin{equation}\label{eq:bounds-entropy}
    c_1(\tilde\rho^2+\tilde m^2)\leq \tilde\eta(\tilde\rho,\tilde m)\leq c_2(\tilde\rho^2+\tilde m^2)\;.
\end{equation}
The constant values can be estimated as follows: rewrite $\tilde\eta$ as 
 a sum, $\tilde\eta =  g(\tilde\rho)+h(\tilde\rho,\tilde m)$  with 
 \begin{equation*}
    g(\tilde\rho):=(\tilde\rho+\bar\rho)(\eps(\tilde\rho+\bar\rho)-\eps(\bar\rho))-\tilde\rho\,\bar\rho\eps'(\bar\rho)\,,\qquad h(\tilde\rho,\tilde m):=\frac{1}{2}\frac{\tilde m^2}{\tilde\rho+\bar\rho}\,.
 \end{equation*}
Noticing that $g(0)=g'(0)=0$ and $g''(\tilde\rho)= \frac{p'(\tilde\rho+\bar\rho)}{\tilde\rho+\bar\rho}$, we can choose
\begin{equation}\label{def:c1c2}
c_1 = \min\left\{\min_{[\bar\rho_{inf},\bar\rho_{sup}]}\frac{p'(\rho)}{\rho},\,\frac{1}{2 \bar\rho_{sup} }\right\}
\,,\qquad
c_2 = \max\left\{\max_{[\bar\rho_{inf},\bar\rho_{sup}]}\frac{p'(\rho)}{\rho},\,\frac{1}{2 \bar\rho_{inf} }\right\}.
\end{equation} 

By denoting with $\tilde q(\tilde\rho, \tilde m)$ the corresponding entropy flux, 
then the entropy weak solutions to~\eqref{eq:final-system} satisfy the inequality in the sense of distributions: 
\begin{equation}\label{eq:entropy-inequality}
    \partial_t\tilde\eta(\tilde\rho,\tilde m)+\partial_x\tilde q(\tilde\rho,\tilde m) \leq \frac{\tilde m}{\tilde\rho+\bar\rho}\left( -2\bar K\bar\rho\, \tilde m + \tilde \caG[\tilde\rho+\bar\rho,\tilde m] \right).
\end{equation}
Integrating over the torus $\bbt$ the above inequality, we obtain
\begin{equation}\label{eq:entropy-inequality-2}
    \frac{d}{dt}\int_\bbt \tilde\eta(\tilde\rho,\tilde m)\,dx \leq \int_\bbt \left\{-2\bar K\bar\rho\frac{\tilde m^2}{\tilde\rho+\bar\rho}+\frac{\tilde m}{\tilde\rho+\bar\rho}\tilde \caG[\tilde\rho+\bar\rho,\tilde m]\right\}\,dx \le 0
\end{equation}
where, in the last inequality, we employed \eqref{eq:property-G} for $\rho=\tilde\rho+\bar\rho$ and $m=\tilde m$.
Therefore we rewrite inequality~\eqref{eq:entropy-inequality-2} as
\begin{equation*}
        \frac{d}{dt}\int_\bbt \tilde\eta(\tilde\rho,\tilde m)\,dx + 2\bar K\bar\rho\int_\bbt \frac{m^2}{\tilde \rho+\bar\rho}\,dx + \mathcal{H}(t)\leq
        0\,,
    \end{equation*}
that will later yield the time decay of the entropy $\tilde\eta$. 

\section{Decay Estimate}\label{S3}
\setcounter{equation}{0}
This section is mainly devoted to the proof of Theorem~\ref{th:energy-estimates}. In Subsection~\ref{S3.1}, we compare the energy functional $\mathcal{E}(t)$ with the $L^2$ norm and show that it is also controlled by the modified kinetic energy $\mathcal{Z}(t)$. In Subsection~\ref{S3.2}, we prove that $\mathcal{E}(t)$ is nonincreasing. Finally, in Subsection~\ref{S3.3}, we conclude with the proof of Theorem~\ref{th:energy-estimates} .

\subsection{The energy functional}\label{S3.1}
We recall the energy functional introduced in \eqref{def:energy-function}: 
\begin{equation*}
    \mathcal{E}(t)=\int_\bbt \left[ \sigma\tilde\eta +\Phi^2-\frac{1}{\bar K\bar\rho}\Phi \tilde m\right]dx\,,
\end{equation*}
where $(\tilde\rho,\tilde m)$ is an entropy weak solution to~\eqref{eq:final-system} with a strictly convex entropy $\tilde\eta$ satisfying~\eqref{eq:bounds-entropy}, $\Phi$ is the potential functional given at~\eqref{def:Phi-function} and $\sigma$ a positive constant to be determined.

In the next lemma we show that, for $\sigma$ sufficiently large, the functional $ \mathcal{E}(t)$ is equivalent to the $L^2$ norm of $(\tilde\rho,\tilde m)$. 

\begin{lemma} \label{S3:lemma1} Assume \eqref{eq:bounds-rho tilde}, \eqref{def:c1c2} and
\begin{equation}\label{bound-on-sigma_1}
\sigma> \sigma_1:=\frac{1}{c_1\bar K\bar\rho}    \;.
\end{equation} 
Then the following inequalities hold: 
    \begin{equation}\label{eq:E-equiv-L2}
        \lambda_1\int_\bbt (\tilde\rho^2(x,t)+\tilde m^2(x,t) )\,dx 
\le \mathcal{E}(t)\le
\lambda_2\int_\bbt (\tilde\rho^2(x,t)+\tilde m^2(x,t))\,dx
    \end{equation}
for all $t\in[0,T^*)$, for some positive constants $\lambda_1$ and $\lambda_2$ depending on $\sigma$, $\bar K$, $\bar\rho$, $c_1$ and $c_2$.
\end{lemma}
\begin{proof}
Using~\eqref{eq:bounds-rho tilde}, \eqref{eq:bounds-phi} and~\eqref{eq:bounds-entropy}, we get
\begin{align}\nonumber
    \mathcal{E}(t) & = ~\sigma\int_\bbt \tilde \eta\,dx + \int_\bbt \left\{ \Phi^2-\frac{1}{\bar K\bar\rho}\Phi \,\tilde  m\right\}\,dx \\ \nonumber
    &  \geq ~\sigma c_1\int_\bbt (\tilde \rho^2+\tilde  m^2)\,dx + \frac{1}{\bar K\bar\rho}\int_\bbt -\Phi \,\tilde  m\,dx \nonumber \\
    & \geq \sigma c_1\int_\bbt(\tilde \rho^2+\tilde  m^2)\,dx +\frac{1}{2\bar K\bar\rho}\int_\bbt \left\{-\Phi^2-\tilde  m^2 \right\}\,dx\\ \nonumber
    & \geq~ \sigma c_1\int_\bbt (\tilde \rho^2+\tilde  m^2)\,dx +\frac{1}{2\bar K\bar\rho}\int_\bbt \left\{-2\tilde \rho^2-\tilde  m^2\right\}\,dx\\ \label{E-bound-below}
    &= ~\left(\sigma c_1-\frac{1}{\bar K\bar\rho}\right)\int_\bbt \tilde \rho^2\,dx + \left(\sigma c_1-\frac{1}{2\bar K\bar\rho}\right)\int_\bbt \tilde m^2\,dx\,.
\end{align}
Thus, for $\sigma$ satisfying~\eqref{bound-on-sigma_1}, we get the bound of $\mathcal{E}(t)$ from below. 

Concerning the estimate of $\mathcal{E}(t)$ from above, we have
 \begin{align}
    \mathcal{E}(t) & \leq c_2\sigma\int_\bbt (\tilde \rho^2+\tilde m^2)\,dx +  \int_\bbt\Phi^2\,dx +\frac{1}{2\bar K\bar \rho}\int_\bbt (\Phi^2+\tilde m^2)\,dx \nonumber\\
    &\leq c_2\sigma\int_\bbt (\tilde \rho^2+\tilde m^2)\,dx + \left[ 2\int_\bbt\tilde \rho^2\,dx +\frac{1}{2\bar K\bar \rho}\int_\bbt (2\tilde \rho^2+\tilde m^2)\,dx\right] \nonumber\\
    &\quad = \left( c_2\sigma+2+\frac{1}{\bar K\bar\rho} \right)\int_\bbt \tilde \rho^2\,dx + \left( c_2\sigma+\frac{1}{2\bar K\bar\rho} \right)\int_\bbt \tilde m^2\,dx \label{E-above}
\end{align}
using again~\eqref{eq:bounds-phi} and~\eqref{eq:bounds-entropy}.
Combining~\eqref{E-bound-below} and~\eqref{E-above}, we arrive at
$$
\left(\sigma c_1-\frac{1}{\bar K\bar\rho}\right)\int_\bbt (\tilde \rho^2(x,t)+\tilde m^2(x,t) )\,dx 
\le \mathcal{E}(t)\le
\left( c_2\sigma+2+\frac{1}{\bar K\bar\rho} \right)\int_\bbt (\tilde \rho^2(x,t)+\tilde m^2(x,t))\,dx, 
$$
$\forall\, t\in[0,T^*)$.
By setting
\begin{equation*}
    \lambda_
1 :=\sigma c_1-\frac{1}{\bar K\bar\rho}\,,\qquad \lambda_
2:=c_2\sigma+2+\frac{1}{\bar K\bar\rho}\,,
\end{equation*}
we conclude that  
$0<\lambda_1<\lambda_2$ if $\sigma$ satisfies~\eqref{bound-on-sigma_1}, and that \eqref{eq:E-equiv-L2} holds. The proof is complete.
\end{proof}
Now, we recall the modified kinetic functional $\mathcal Z(t)$ defined in \eqref{def:Z-function}
with $\kappa$ and $\beta$ satisfying \eqref{kappa-bound} and, in the next proposition, we show that, under appropriate choices of the parameters, the energy $\mathcal E(t)$ is controlled by $\mathcal Z(t)$.

\begin{proposition}\label{prop:3.3}
If \eqref{def:c1c2}, \eqref{kappa-bound},~\eqref{bound-on-sigma_1} hold and 
\begin{equation}\label{bound-on-sigma_2}
\sigma>\sigma_2:=\frac{1}{\kappa\bar K}\,,
\end{equation}
then 
\begin{equation}\label{Z-E}
    \mathcal Z(t)\geq 2c \,\mathcal E(t)\,,\qquad t\in[0,T^*)
\end{equation}
where
\begin{equation}
c :=\min\left\{\frac{1}{\bar\rho_{sup}}\frac{\bar\rho(\sigma\kappa \bar K -1)}{1+2\bar K\bar\rho c_2\sigma},\frac{1}{2}\frac{\beta\bar K\bar\rho}{\bar K\bar\rho c_2\sigma+2\bar K\bar\rho+1} \right\} >0\;.\label{eq:constants-c}
\end{equation}
\end{proposition}

\begin{remark}
    If the quantity $\bar K$ in \eqref{hyp-on-K} goes to $0$, then also the constant $c$ defined in \eqref{eq:constants-c} goes to $0$. Therefore, it is fundamental to consider this result in terms of a convolution kernel that is equal to the sum of a constant value and a function in $L^1(\bbt)$.
\end{remark}

\begin{proof} The function $\mathcal Z(t)$ satisfies
\begin{align*}
    \mathcal{Z}(t) & \geq \left( \sigma\kappa-\frac{1}{\bar K}\right)\frac{1}{\bar\rho_{sup}}
    \int_\bbt \tilde m^2\,dx + \beta \int_\bbt \tilde\rho^2\,dx\,,
\end{align*}
from~\eqref{eq:bounds-rho tilde}. Now first recalling the upper bound~\eqref{E-above} for the energy $\mathcal{E}(t)$ 
and then taking $c$ as in \eqref{eq:constants-c} the following inequalities are satisfied
\begin{equation*}
  2c \left( c_2\sigma+\frac{1}{2\bar K\bar\rho} \right)\le \left(\sigma\kappa-\frac{1}{\bar K}\right)\frac{1}{\bar\rho}_{sup}
  \,,\qquad 
  2c\left(  c_2\sigma+2+\frac{1}{\bar K\bar\rho} \right)\le \beta \,,
\end{equation*}
Thus estimate~\eqref{Z-E} holds true and the proof is complete.
\end{proof}

\subsection{A Gronwall estimate for the energy}\label{S3.2}
This subsection is devoted in study of the evolution of the energy functional and it is shown that is nonincreasing in time, i.e., 
\begin{equation*}
    \mathcal{E}(t_2)-\mathcal{E}(t_1)\leq 0,\qquad \forall\, t_2,\,t_1\in[0,T^*)\quad\text{for}\quad t_2>t_1.
\end{equation*}
The following theorem states the key properties of $\mathcal{E}(t)$ that allows us to prove Theorem~\ref{th:L2-decay}. More precisely, this provides the time decay of the energy functional $\mathcal{E}(t)$ in a Gronwall-type setting that would later enable us to extract an uniform bound on the $L^2$ norm of solutions $(\tilde\rho,\tilde m)$. 

\begin{theorem}\label{th:energy-estimates}
    Under the assumptions of Theorem \ref{th:L2-decay}, 
if $\sigma$ is chosen sufficiently large and \eqref{kappa-bound} hold,
   then the functionals $\mathcal H$, $\mathcal E$ and $\mathcal Z$ given at~\eqref{def:H-function}, \eqref{def:energy-function} and \eqref{def:Z-function} respectively, satisfy
    \begin{equation}\label{eq:inequality-EHZ1}
        \mathcal E(t_2)-\mathcal E(t_1) + \int_{t_1}^{t_2} \left[\sigma \mathcal H(t) +  \mathcal Z(t) \right]\,dt \leq 0
    \end{equation}
and  
    \begin{equation}\label{eq:gronw-energy}
        \mathcal{E}(t) + \sigma\int_0^t \mathcal H(s)e^{-2c(t-s)}\,ds \leq \mathcal E(0)e^{-2ct}
    \end{equation}
for any $t,\,t_1,\,t_2\in[0, T^*)$ with $t_2>t_1$, where $c$ is given in \eqref{eq:constants-c}. 
\end{theorem}

\begin{proof} 
To begin with, we multiply the second equation of system~\eqref{eq:final-system} by the potential $\Phi(t)$, 
then add the term 
$\tilde m\partial_t\Phi+\left( \frac{\tilde m^2}{\tilde\rho+\bar\rho}+p(\tilde\rho+\bar\rho) \right)\partial_x\Phi$ 
to get
\begin{align}
    &\partial_t(\Phi \tilde m)+\partial_x \left(\Phi\left(\frac{\tilde m^2}{\tilde \rho+\bar\rho}+p(\tilde\rho+\bar\rho)\right)\right) \nonumber\\
    &\qquad =\Phi(-2\bar K \bar\rho \,\tilde m + \tilde \caG[\tilde\rho+\bar\rho,\tilde m])-\tilde m^2+\tilde\rho\left( \frac{\tilde m^2}{\tilde\rho+\bar\rho}+p(\tilde\rho+\bar\rho) \right)\nonumber\\
    &\qquad = -2\Phi\bar K \bar\rho \tilde m + \Phi\tilde \caG[\tilde\rho+\bar\rho,\tilde m] -\tilde m^2 \frac{\bar \rho}{\tilde\rho+\bar\rho}+\tilde\rho \,p(\tilde\rho+\bar\rho)\label{eq:Phi-m}\,,
\end{align}
using \eqref{eq:partial-x-phi}--\eqref{eq:partial-t.phi}. Furthermore, multiplying \eqref{eq:partial-t.phi} by $2\Phi$, we obtain 
\begin{equation}\label{eq:Phi^2}
\partial_t\Phi^2 = -2\tilde m\,\Phi.
\end{equation}
Now, we write inequality \eqref{eq:entropy-inequality} and equations~\eqref{eq:Phi-m} and \eqref{eq:Phi^2} in distributional sense. 
Indeed, for any nonnegative periodic test function $\varphi\in C^1(\bbt\times (0,T^*))$ with support contained in $\bbt\times [t_1,t_2]$ 
for some $0<t_1<t_2<T^*$, we combine these three integral relations to reach the inequality 
\begin{align}
    \int_0^{T^*}\int_\bbt & \left(\sigma\tilde\eta +\Phi^2 -\frac{1}{\bar K\bar\rho}\Phi \,\tilde m\right)\partial_t\varphi\,dx\,dt  \nonumber\\ 
    & \geq \int_0^{T^*}\int_\bbt -\sigma\tilde q\partial_x\varphi +\sigma\frac{\tilde m}{\tilde\rho+\bar\rho}\left( 2\bar K\bar\rho \,\tilde m -\tilde \caG[\tilde\rho+\bar\rho,\tilde m] \right)\varphi\,dx\,dt\nonumber\\
    & \quad+\int_0^{T^*}\int_\bbt \frac{1}{\bar K\bar\rho}\left( \Phi\left(\frac{\tilde m}{\tilde\rho+\bar\rho}+p(\tilde\rho+\bar\rho)\right) \right)\partial_x \varphi \,dx\,dt\nonumber\\
    &\quad+\int_0^{T^*}\int_\bbt -\frac{1}{\bar K\bar\rho}\left( 2\Phi\bar K\bar\rho \,\tilde m - \Phi\,\tilde \caG[\tilde\rho+\bar\rho, \tilde m] + \tilde m^2\frac{\bar\rho}{\tilde\rho+\bar\rho} -\tilde\rho \,p(\tilde\rho+\bar\rho)\right)\varphi\,dx\,dt\nonumber\\
    &\quad +\int_0^{T^*}\int_\bbt 2\tilde m\,\Phi\,\varphi\,dx\,dt\,.\label{eq:ineq-energy-1}
\end{align}
Notice that the right hand side of the inequality above is equal to
\begin{align}
    \int_0^{T^*}\int_\bbt & \left[ \frac{1}{\bar K\bar\rho}\left( \Phi\frac{\tilde m}{\tilde\rho+\bar\rho}+\Phi  \,p(\tilde\rho+\bar\rho) \right) - \sigma \tilde q\right]\partial_x\varphi\,dx\,dt+\left(2\sigma\bar K\bar\rho-\frac{1}{\bar K}\right)\int_0^{T^*}\int_\bbt \frac{\tilde m^2}{\tilde\rho+\bar\rho}\varphi\,dx\,dt\nonumber\\
    &\quad+ \int_0^{T^*}\int_\bbt\left[ \frac{1}{\bar K\bar\rho} \left( \Phi\,\tilde \caG[\tilde\rho+\bar\rho,\tilde m]+\tilde\rho\, p(\tilde\rho+\bar\rho) \right) - \sigma\frac{\tilde m}{\tilde\rho+\bar\rho}\tilde \caG[\tilde\rho+\bar\rho,\tilde m]\right]\varphi\,dx\,dt\,.\label{eq:distrib-ineq-energy}
\end{align}
Given $0<t_1<t_2<T^*$, let $\eps>0$ be small such that $t_1-\eps\geq0$.
We now select the test function to be  $\varphi(x,t)=\chi(x)h_\eps(t)\in C^1(\bbt\times (0,T^*))$ with
\begin{align*}
    \chi(x) & :=1,\qquad x\in\bbt\\
    h_\eps(t) & := 
    \begin{cases}
        0, & 0\leq t<t_1-\eps,\\
        \eps^{-1}(t-t_1)+1,\quad  &  t_1-\eps \leq t < t_1,\\
        1, & t_1 \leq t \leq t_2,\\
        \eps^{-1}(t_2-t) + 1, & t_2 < t \leq t_2+\eps,\\
        0, & t> t_2+\eps,
    \end{cases}\qquad t\in (0,T^*)
\end{align*}
and evaluate~\eqref{eq:distrib-ineq-energy} for the test function $\phi(x,t)=\chi(x)h_\eps(t)$ to get
\begin{align*}
    & \left( 2\sigma\bar K\bar\rho-\frac{1}{\bar K}\right)\int_{t_1-\eps}^{t_2+\eps}\int_\bbt \frac{\tilde m^2}{\tilde\rho+\bar\rho}h_\eps(t)\,dx\,dt\\
    & \qquad+ \int_{t_1-\eps}^{t_2+\eps}\int_\bbt \left[ \frac{1}{\bar K\bar\rho}\left( \Phi\, \tilde \caG[\tilde\rho+\bar\rho, \tilde m] + \tilde\rho\, p(\tilde\rho+\bar\rho) \right) - \sigma\frac{\tilde m}{\tilde\rho+\bar\rho}\tilde \caG [\tilde\rho+\bar\rho, \tilde m] \right]h_\eps(t)\,dx\,dt.
\end{align*}
Using the definition~\eqref{def:energy-function} of $\mathcal{E}(t)$, inequality~\eqref{eq:ineq-energy-1}  and the above computations, we now let $\eps\to 0+$ to obtain
\begin{align}
    \mathcal{E}(t_2) & -\mathcal{E}(t_1) \leq  -\left( 2\sigma\bar K\bar\rho-\frac{1}{\bar K}\right)\int_{t_1}^{t_2}\int_\bbt \frac{\tilde m^2}{\tilde\rho+\bar\rho}\,dx\,dt\nonumber\\
    & \quad- \int_{t_1}^{t_2}\int_\bbt \left[ \frac{1}{\bar K\bar\rho}\left( \Phi \,\tilde \caG[\tilde\rho+\bar\rho, \tilde m] +\tilde \rho \,p(\tilde\rho+\bar\rho) \right) - \sigma\frac{\tilde m}{\tilde\rho+\bar\rho}\tilde \caG [\tilde\rho+\bar\rho,\tilde m] \right]\,dx\,dt.\nonumber
\end{align}
The last inequality can be rewritten as
\begin{align}\label{eq:en H ineq}
    \mathcal E(t_2)-\mathcal E(t_1) + \sigma\int_{t_1}^{t_2} \mathcal H(t)\,dt \leq -\left( 2\sigma\bar K\bar\rho-\frac{1}{\bar K}\right)\int_{t_1}^{t_2}\int_\bbt \frac{\tilde m^2}{\tilde\rho+\bar\rho}\,dx\,dt\nonumber\\
    - \int_{t_1}^{t_2}\int_\bbt \left[ \frac{1}{\bar K\bar\rho}\left( \Phi \tilde \caG[\tilde\rho+\bar\rho, \tilde m] +\tilde \rho \,p(\tilde\rho+\bar\rho) \right)\right]\,dx\,dt\;,
\end{align}
using~\eqref{eq:property-G} for $\rho=\tilde\rho+\bar\rho$ and $m=\tilde m$.

Recalling \eqref{hyp-on-p}, we can use Taylor expansion with Lagrange remainder
$$p(\tilde\rho+\bar\rho) = p(\bar\rho)+ p'(\xi)\tilde\rho    
$$
for some $\xi\in (\bar\rho_{inf},\bar\rho_{sup})$. Hence 
$$\tilde\rho p(\tilde\rho+\bar\rho) =\tilde\rho p(\bar\rho)+ p'(\xi)\tilde\rho^2 \ge \tilde\rho p(\bar\rho)+ C\tilde\rho^2\;,
$$
with 
$$C:= \inf_{r\in (\bar\rho_{inf},\bar\rho_{sup})} p'(r) >0\;.$$
Since $\int_{\bbt}\tilde \rho \,dx=0$, then
$$\int_\bbt \tilde\rho p(\tilde\rho+\bar\rho)\,dx \ge C \int_\bbt\tilde \rho^2 \,dx,
$$
and in view of the above analysis, inequality~\eqref{eq:en H ineq} takes the form
\begin{align}\label{eq:E-H}
    \mathcal E(t_2)-\mathcal E(t_1) + \sigma\int_{t_1}^{t_2} \mathcal H(t)\,dt \leq -\left( 2\sigma\bar K\bar\rho-\frac{1}{\bar K}\right)\int_{t_1}^{t_2}\int_\bbt \frac{\tilde m^2}{\tilde\rho+\bar\rho}\,dx\,dt\nonumber\\
    - \int_{t_1}^{t_2}\int_\bbt \left[ \frac{1}{\bar K\bar\rho}\left( \Phi \,\tilde \caG[\tilde\rho+\bar\rho,\tilde m] + C \tilde\rho^2 \right)\right]\,dx\,dt.
\end{align}
Next, by means of the definition of $\mathcal{Z}$ in \eqref{def:Z-function} we find
\begin{align}
    & \mathcal{E}(t_2)-\mathcal{E}(t_1) + \int_{t_1}^{t_2}\left[\sigma \mathcal{H}(t)+ \mathcal{Z}(t)\right]\,dt \nonumber\\
   & \qquad \le 
       - \sigma(2\bar K\bar\rho-\kappa)\int_{t_1}^{t_2}\int_\bbt\frac{\tilde m^2}{\tilde\rho+\bar\rho}\,dx  + \left(\beta -\frac{C}{\bar K\bar\rho}\right)\int_{t_1}^{t_2}\int_\bbt \tilde\rho^2\,dx \nonumber \\
        &\qquad \quad - \frac{1}{\bar K\bar\rho}\int_{t_1}^{t_2}\int_\bbt\Phi\, \tilde \caG(\tilde\rho+\bar\rho,   \tilde m)\,dx.\label{E+sH+Z}
\end{align}
Because of \eqref{kappa-bound}, we know that the first two terms are negative. 
Hence, it remains to control 
the last term $-\frac{1}{\bar K\bar\rho}\int_\bbt\Phi(x)\,\tilde \caG[\tilde\rho+\bar\rho,\tilde m](x)\,dx$. 
This one can be split into two integrals as follows:
\begin{align*}
     - & \frac{1}{\bar K\bar\rho} \int_\bbt\Phi(x)\tilde \caG[\tilde\rho+\bar\rho,\tilde m](x)\,dx \\
    & = -\frac{1}{\bar K\bar\rho}\int_\bbt\int_\bbt \tilde K (x-x')\Phi(x)[\tilde\rho(x) \tilde m(x')-\tilde\rho(x')\tilde m(x)]\,dx\,dx'\\
    &\qquad-\frac{1}{\bar K}\int_\bbt\int_\bbt \tilde K (x-x')\Phi(x)[\tilde m(x') - \tilde m(x)]\,dx\,dx'=:J_1+J_2
\end{align*}
Noting that, by the symmetry of the integrand, it holds
\begin{equation*}
    \frac{1}{\bar K\bar\rho}\int_\bbt\int_\bbt \tilde K(x-x')[\tilde \rho(x)\tilde m(x')-\tilde\rho(x')\tilde m(x)]\,dxdx'=0, 
\end{equation*}
then from~\eqref{def:potential_2}, we can reduce term $J_1$ to
\begin{align*}
J_1=- & \frac{1}{\bar K\bar\rho}\int_\bbt\int_\bbt \tilde K (x-x')\psi(x)[\tilde\rho(x) \tilde m(x')-\tilde\rho(x') \tilde m(x)]\,dx\,dx'\;.
\end{align*}
By estimating, we get 
\begin{align*}
 J_1=   -\frac{1}{\bar K\bar\rho} & \int_\bbt\int_\bbt \tilde K (x-x')\psi(x)[\tilde \rho(x) \tilde m(x')-\tilde \rho(x')\tilde m(x)]\,dx\,dx'\\
    \leq & \frac{1}{\bar K\bar\rho}\int_\bbt\int_\bbt |\psi(x)|\tilde K(x-x')|\tilde\rho(x) \tilde m(x') - \tilde\rho(x') \tilde m(x)|\,dxdx'\\
    \leq & \frac{1}{\bar K\bar\rho}\int_\bbt\int_\bbt (\bar\rho_{sup}-\bar\rho)\tilde  K(x-x')(|\tilde \rho(x) \tilde m(x')|+|\tilde\rho(x') \tilde m(x)|)\,dxdx'\\
    = & \frac{2(\bar\rho_{sup}-\bar\rho)}{\bar K\bar\rho}\int_\bbt\int_\bbt \tilde K(x-x')|\tilde\rho(x) \tilde m(x')|\,dxdx'\\
    \leq & \frac{(\bar\rho_{sup}-\bar\rho)}{\bar K\bar\rho}\int_\bbt\int_\bbt \tilde K(x-x')[\frac{1}{\alpha^2}\tilde\rho^2(x)+\alpha^2 \tilde m^2(x')]\,dxdx'\\
    = & \frac{\|\tilde K\|_{L^1}(\bar\rho_{sup}-\bar\rho)}{\bar K\bar\rho}\int_\bbt[\frac{1}{\alpha^2}\tilde\rho^2(x)+\alpha^2 \tilde m^2(x)]\,dx,
\end{align*}
for a parameter $\alpha>0$ to be determined. About $J_2$, we proceed in a similar fashion and also use~\eqref{eq:bounds-phi} to obtain
\begin{align*}
    J_2 =~&-\frac{1}{\bar K}  \int_\bbt\int_\bbt \tilde K (x-x')\Phi(x)[\tilde m(x') - \tilde m(x)]\,dx\,dx'\\ 
    =~& -\frac{1}{\bar K}\int_\bbt\int_\bbt \tilde K(x-x')\Phi(x)\tilde m(x')\,dxdx'
   + \frac{1}{\bar K}\int_\bbt\int_\bbt \tilde K(x-x')\Phi(x)\tilde m(x)\,dxdx'\\
     \leq ~& \frac{1}{2\bar K}\int_\bbt\int_\bbt \tilde K(x-x')[\frac{1}{\alpha^2}\Phi^2(x)+\alpha^2\tilde m^2(x')]\,dxdx'\\
    &\qquad \qquad + \frac{1}{2\bar K}\int_\bbt\int_\bbt \tilde K(x-x')[\frac{1}{\alpha^2}\Phi^2(x)+\alpha^2\tilde m^2(x)]\,dxdx'\\
   =~ &\frac{1}{\bar K\alpha^2}\int_\bbt\int_\bbt \tilde K(x-x')\Phi^2(x)\,dxdx' + \frac{\alpha^2}{\bar K}\int_\bbt\int_\bbt \tilde K(x-x') \tilde m^2(x)\,dxdx'\\
    =  ~&  \frac{\| \tilde K\|_{L^1}}{\bar K\alpha^2}\int_\bbt\Phi^2(x)\,dx + \frac{\alpha^2\| \tilde K\|_{L^1}}{\bar K}\int_\bbt  \tilde m^2(x)\,dx\\
  \leq ~&  \frac{2\| \tilde K\|_{L^1}}{\alpha^2\bar K}\int_\bbt\tilde\rho^2(x)\,dx + \frac{\alpha^2\| \tilde K\|_{L^1}}{\bar K}\int_\bbt  \tilde m^2(x)\,dx.
\end{align*}
Combining the above with estimate~\eqref{E+sH+Z}, we arrive at
\begin{align}
    \mathcal{E}(t_2)  &-\mathcal{E}(t_1)  + \int_{t_1}^{t_2}\left[\sigma \mathcal{H}(t) + \mathcal{Z}(t)\right]\,dt \nonumber\\
    & \leq \left[- \frac{\sigma(2\bar K\bar\rho-\kappa)}{\bar\rho_{sup}}+\frac{\|\tilde K\|_{L^1}}{\bar K}\alpha^2\left(  1+\frac{\bar\rho_{sup}-\bar\rho}{\bar\rho}\right)\right]\int_{t_1}^{t_2}\int_\bbt \tilde m^2\,dx \nonumber\\
    & \qquad + \left[\frac{\|\tilde K\|_{L^1}}{\bar K\alpha^2}\left( 2+\frac{\bar\rho_{sup}-\bar\rho}{\bar\rho} \right) + \beta - \frac{C}{\bar K\bar\rho}\right]\int_{t_1}^{t_2}\int_\bbt\tilde\rho^2\,dx\,. \label{eq:E+sigma_H+Z}
\end{align}
Having $\kappa$ and $\beta$ satisfying~\eqref{kappa-bound}, we choose $\alpha$ and $\sigma$ 
so that the right hand side of inequality~\eqref{eq:E+sigma_H+Z} to be non-positive. 
So, we first select $\alpha>0$ to fulfill
\begin{equation}\label{cond-on-alpha}
\alpha^2\geq \frac{\| \tilde K \|_{L^1(\bbt)}(\bar\rho+\bar\rho_{sup})}{C-\beta\bar K\bar\rho}
\end{equation}
and then $\sigma>\sigma_3$, with
\begin{equation}\label{bound-on-sigma_3}
\sigma_3:=\alpha^2 \frac{\| \tilde K \|_{L^1(\bbt)}}{\bar\rho\bar K}\frac{\bar\rho_{sup}^2}{2\bar K\bar\rho-\kappa}\;.
\end{equation}
Thus,~\eqref{eq:inequality-EHZ1} holds true. Moreover, thanks to Proposition \ref{prop:3.3}, it holds true~\eqref{eq:gronw-energy} by Gronwall Lemma. 
In summary, if $\sigma>\max\{\sigma_1,\sigma_2,\sigma_3\}$ given at ~\eqref{bound-on-sigma_1}, \eqref{bound-on-sigma_2} and \eqref{bound-on-sigma_3}, that is
\begin{equation}
\sigma  > \max\left\{ \alpha^2 \frac{\| \tilde K \|_{L^1(\bbt)}}{\bar\rho\bar K}\frac{\bar\rho_{sup}^2}{2\bar K\bar\rho-\kappa}, \frac 1 {\kappa \bar K}, \frac{1}{c_1\bar K\bar\rho}\right\},\label{eq:constants}    
\end{equation}
then both~\eqref{eq:inequality-EHZ1} and ~\eqref{eq:gronw-energy} follow. The proof of Theorem~\ref{th:energy-estimates} is complete.
\end{proof}

\subsection{Proof of Theorems \ref{th:L2-decay} and~\ref{thm 1:L2-decay}}\label{S3.3}
At this point, we have all the ingredients needed to prove Theorem~\ref{th:L2-decay} and its counterpart in the original variables $(\rho,m)$, Theorem~\ref{thm 1:L2-decay}.

\begin{proof}[Proof of Theorem \ref{th:L2-decay}]
Having an entropy weak solution $(\tilde\rho,\tilde m)$ to \eqref{eq:final-system}--\eqref{eq:init-data tilde} 
that satisfy Theorem~\ref{th:energy-estimates} for appropriate choice of parameters, we can show decay in $L^2$ norm on $\bbt$ as follows.

By Lemma~\ref{S3:lemma1}, the energy is bounded from below by the $L^2$ norm of $(\tilde\rho,\tilde m)$: in particular, there exists $\sigma_0>0$ such that
\begin{equation}\label{S3.3 e L2}
   \mathcal{E}(t)\ge \left(\sigma c_1-\frac{1}{\bar K\bar\rho}\right)\int_\bbt \left(\tilde\rho^2+\tilde m^2\right)\,dx \geq \sigma_0 \int_\bbt \left(\tilde\rho^2 +\tilde m^2\right)\,dx
\end{equation}
for $\sigma$ satisfying~\eqref{eq:constants} and
\begin{equation}\label{bound-on-sigma_4}
\sigma>\sigma_4:=\frac{1}{c_1}\left(\frac{1}{\bar K\bar\rho} + {\sigma_0}\right) \;.
\end{equation}
Here, we note  that $\sigma_4= \sigma_1 + \frac {\sigma_0} {c_1} >\sigma_1$.

Therefore, it is straightforward to bound the $L^2$ norm of $\tilde\rho$ and $\tilde m$ 
\begin{align}\label{S3.3: L2 bound init E}
\| (\tilde\rho(x,t), \tilde m(x,t)\|_{L^2(\bbt)} \leq \frac1{\sqrt{\sigma_0}}\sqrt{\mathcal{E}(t)}\le\frac1{\sqrt{\sigma_0}}
    \sqrt{\mathcal{E}(0)}e^{-ct} ,
\end{align}
for $t\in[0,T^*)$, using~\eqref{eq:gronw-energy} and~\eqref{S3.3 e L2} for $\sigma>\max\{\sigma_2,\,\sigma_3,\,\sigma_4\}$ 
and get the exponential in time decay. Moreover, the quantity $\sqrt{\mathcal{E}(0)}$ is bounded by the $L^2$ norm of the initial data. 
Indeed, recalling \eqref{E-above}, we have
\begin{align*}
    \mathcal{E}(0) &  \leq \mu^2\left( \int_\bbt \tilde\rho^2_0\,dx + \int_\bbt \tilde m^2_0\,dx\right) =\mu^2 \| (\tilde\rho_0,\tilde m_0)\|_{L^2}^2
\end{align*}
with $\mu^2=c_2\sigma + 2 + \frac{1}{\bar K\bar \rho}$. 
Thus, combining with~\eqref{S3.3: L2 bound init E}, the proof of estimate \eqref{prop:three} is complete.
\end{proof}

\begin{proof}[Proof of Theorem \ref{thm 1:L2-decay}]
The proof is immediate via the change of variables~\eqref{def:tilde-var} in Section~\ref{S2} and Theorem~\ref{th:L2-decay}.
\end{proof}

\small{

}
\end{document}